\pdfoutput=1
\documentclass[12pt]{article}

\usepackage[reqno]{amsmath} 
\usepackage{amssymb}
\usepackage{amsthm}
\usepackage{enumerate}  
\usepackage{dsfont}
\usepackage{sectsty}
\usepackage{graphicx}
\usepackage{hyperref}
\sectionfont{\fontsize{14}{15}\selectfont}

\theoremstyle{plain}
\newtheorem{conjecture}{Conjecture}
\swapnumbers
\newtheorem{thm}{Theorem}[section]
\newtheorem{lem}[thm]{Lemma}
\newtheorem{cor}[thm]{Corollary}

\theoremstyle{definition}

\newcommand{\ii}{\mathrm{i}}
\newcommand{\e}{\mathrm{e}}
\newcommand{\ee}{\mathbf{e}}

\newcommand{\zz}{\mathbf{z}}
\newcommand{\yy}{\mathbf{y}}

\newcommand{\C}{\mathds{C}}
\newcommand{\R}{\mathds{R}}
\newcommand{\Z}{\mathds{Z}}
\newcommand{\Q}{\mathds{Q}}
\renewcommand{\j}{\mathbf{j}}

\title{No Laplacian Perfect State Transfer in Trees}
\author{Gabriel Coutinho\footnote{Dep. of Combinatorics and Optimization, University of Waterloo. \newline \texttt{ \{gcoutinho, h228liu\}@uwaterloo.ca}. } \ \footnote{Supported by Capes Foundation, Ministry of Education, Brazil.} \and Henry Liu\footnotemark[1] \ \footnote{Supported by NSERC, Canada.}}  

\date{\today}

\begin{document}

\maketitle

\begin{abstract}
We consider a system of qubits coupled via nearest-neighbour interaction governed by the Heisenberg Hamiltonian. We further suppose that all coupling constants are equal to $1$. We are interested in determining which graphs allow for a transfer of quantum state with fidelity equal to $1$. To answer this question, it is enough to consider the action of the Laplacian matrix of the graph in a vector space of suitable dimension. 

Our main result is that if the underlying graph is a tree with more than two vertices, then perfect state transfer does not happen. We also explore related questions, such as what happens in bipartite graphs and graphs with an odd number of spanning trees. Finally, we consider the model based on the $XY$-Hamiltonian, whose action is equivalent to the action of the adjacency matrix of the graph. In this case, we conjecture that perfect state transfer does not happen in trees with more than three vertices.
\end{abstract}

\section{Introduction}

Let $X$ be a simple undirected graph on $n$ vertices, and let $L = L(X)$ denote its Laplacian matrix. For any $u \in V(X)$, we denote its characteristic vector with respect to the ordering of the rows of $L$ by $\ee_u$. The matrix operator
\[U_L(t) = \exp(\ii t L),\]
defined for every real $t \geq 0$, represents a \textit{continuous-time quantum walk} on $X$. We say that $X$ admits \textit{Laplacian perfect state transfer} from a vertex $u$ to a vertex $v$ if there is a time $t \geq 0$ such that
\[U_L(t) \ee_u = \gamma \ee_v\]
for some $\gamma \in \C$. Here we work under the assumption that the graph $X$ is the underlying network of a system of qubits coupled via nearest-neighbour interaction governed by the Heisenberg Hamiltonian with coupling constants equal to $1$. If such a system is initialized in such a way that the state of the qubit located at the vertex $u$ is orthogonal to the states of each of the other qubits, then Laplacian perfect state transfer in the graph between $u$ and $v$ is equivalent to a transfer of state from the qubit at $u$ to the qubit at $v$ with fidelity $1$ (see Kay~\cite{KayPerfectcommunquantumnetworks}). If we choose the $XY$-Hamiltonian instead, perfect state transfer is defined in terms of $U_A(t) = \exp(\ii t A)$, where $A$ is the adjacency matrix of the graph. There are many similarities between both cases, mostly because both $A$ and $L$ are symmetric integer matrices with a positive eigenvector. If $X$ is $k$-regular, then $L = k I - A$, thus $\exp(\ii t A)$ and $\exp(\ii t L)$ differ only by a constant.

The problem of determining which graphs admit adjacency perfect state transfer has received a considerable amount of attention lately. For example, it was solved for paths and hypercubes (see Christandl et al.~\cite{ChristandlPSTQuantumSpinNet2}), circulant graphs (see Ba\v{s}i\'{c} \cite{BasicCirculant}), cubelike graphs (see Cheung and Godsil \cite{GodsilCheungPSTCubelike}) and distance-regular graphs (see Coutinho et al.~\cite{CoutinhoGodsilGuoVanhove}). The effect of certain graph operations was considered in Angeles-Canul et al.~\cite{Angeles-CanulPSTcirculant}, Bachman et al.~\cite{TamonPSTQuotient} and Ge et al.~\cite{GeGreenbergPerezTamonPSTproducts}. Recent surveys are found in Kendon and Tamon \cite{KendonTamon}, and Godsil \cite{GodsilStateTransfer12}. 

Bose et al.~\cite{BoseCasaccinoManciniSeverini} constructed an infinite family of graphs admitting Laplacian perfect state transfer. Kay \cite{KayPerfectcommunquantumnetworks} observed that some Laplacian eigenvalues must be integers in order for Laplacian perfect state transfer to happen. Here we build upon this observation to solve the problem of determining which trees admit Laplacian perfect state transfer.

Our main result is that, except for the path on two vertices, no tree admits Laplacian perfect state transfer. To achieve that, we first show that, in any graph, Laplacian perfect state transfer does not happen between twin vertices sharing one or two common neighbours. Then we apply the Matrix-Tree Theorem coupled with a precise understanding of Laplacian perfect state transfer to show that such phenomenon could only happen in trees between twin vertices. We also show how our methods can also be applied to other classes of graphs.

In the last sections, we study the adjacency perfect state transfer and we will show that no tree with an invertible adjacency matrix admits perfect state transfer.

The importance of our results is twofold. First, we rule out trees as candidates for graphs in which Laplacian state transfer can be achieved at large distances with relatively few edges. To this date, the best known trade off is obtained with iterated cartesian powers of the path on two vertices. We also progress in determining which trees admit (adjacency) perfect state transfer. Secondly, we succeed in exhibiting interesting connections between classical results in algebraic graph theory and relatively modern applications.

Some results of this paper and an extensive elementary treatment of perfect state transfer can be found in Coutinho \cite[PhD Thesis]{CoutinhoPhD}.

\section{Laplacian Perfect State Transfer}

Throughout this paper, $J$ denotes the all-ones matrix of appropriate size except for the all-ones column vector, which we denote by $\j$.

Consider a graph $X$ with adjacency matrix $A$. Let $D$ be the diagonal matrix whose entries are the degrees of the vertices of $X$. Then the \textit{Laplacian matrix} of $X$ is defined as $L = D - A$. The Laplacian matrix of a graph is positive semidefinite, and the multiplicity of $0$ as an eigenvalue is equal to the number of connected components of the graph. The all-ones vector $\j$ is always an eigenvector for $0$. Because $L$ is symmetric, it admits a spectral decomposition into orthogonal projections onto its eigenspaces, which we will typically denote by
\[L = \sum_{r = 0}^d \lambda_r F_r,\]
with the understanding that $0 = \lambda_0 < \lambda_1 < ... < \lambda_d$.

Suppose $u$ and $v$ are two vertices of $X$ with corresponding characteristic vectors $\ee_u$ and $\ee_v$. We say that an eigenvalue $\lambda_r$ is in the \textit{Laplacian eigenvalue support} of $u$ if $F_r \ee_u \neq 0$. We will denote the eigenvalue support of $u$ by $\Lambda_u$. We also define $\Lambda_{uv}^+$ to be such that $\lambda_r \in \Lambda_{uv}^+$ if and only if $F_r \ee_u = F_r \ee_v$, and correspondingly $\Lambda_{uv}^-$ to be such that $\lambda_r \in \Lambda_{uv}^-$ if and only if $F_r \ee_u = - F_r \ee_v$. Note that $\Lambda_{uv}^+$ and $\Lambda_{uv}^-$ are disjoint subsets of $\Lambda_u$.

The following result is a compilation of many well known facts about perfect state transfer, written in the context of the Laplacian matrix. It explicitly states which conditions on the spectral structure of the Laplacian and on the parity of its eigenvalues are equivalent to perfect state transfer.

\begin{thm}\label{pst-char}
Let $X$ be a graph, $u,v \in V(X)$. Let $L$ be the Laplacian matrix of $X$ admitting spectral decomposition $L = \sum_{r = 0}^d \lambda_r F_r$ with $0 = \lambda_0 < \lambda_1 < ... < \lambda_d$. There exist a positive $t \in \R$ and a $\gamma \in \C$ satisfying
\[\exp (\ii t L) \ee_u = \gamma \ee_v\]
if and only if all of the following conditions hold.
\begin{enumerate}[(i)]
\item \label{pst-char:item1} For all $r \in \{0,...,d\}$, $F_r \ee_u = \pm F_r \ee_v$. In particular, $\Lambda_u = \Lambda_v$.
\item \label{pst-char:item2} Elements in $\Lambda_u$ are integers.
\item \label{pst-char:item3} Let $\displaystyle g = \gcd \left\{ \lambda_r  : \lambda_r \in \Lambda_u\right\}$. Then 
\begin{enumerate}[a)]
\item $\lambda_r \in \Lambda_{uv}^+$ if and only if $\dfrac{\lambda_r}{g}$ is even, and
\item $\lambda_r \in \Lambda_{uv}^-$ if and only if $\dfrac{\lambda_r}{g}$ is odd. \qed
\end{enumerate} 
\end{enumerate} 
\end{thm}

\section{Size of the Eigenvalue Support and Twin Vertices}

Let $N(u)$ denote the set of neighbours of $u$ in $X$. We say that $u$ and $v$ are \textit{twin vertices} if
\[N(u) \backslash v = N(v) \backslash u.\]
Note that twin vertices may or may not be adjacent.
\begin{center}
\includegraphics{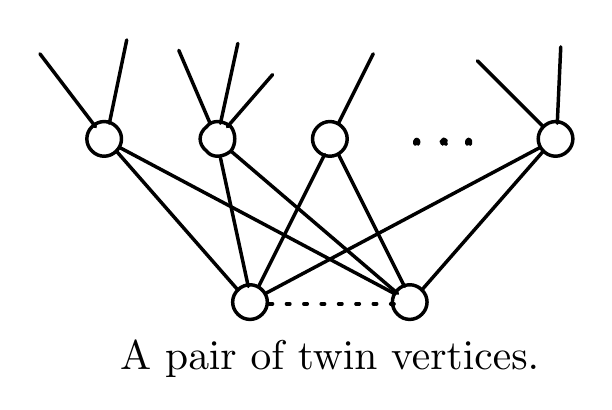}
\end{center}

Our first original contribution to the study of Laplacian perfect state transfer is the following lemma.

\begin{lem}\label{cospec-sizes}
Let $X$ be a connected graph on $n > 2$ vertices admitting Laplacian perfect state transfer between $u$ and $v$. Then ${|\Lambda_{uv}^+| \geq 2}$ and ${|\Lambda_{uv}^-| \geq 1}$. If ${|\Lambda_{uv}^-| = 1}$, then $u$ and $v$ are twins.
\end{lem}
\begin{proof}
Define
\begin{align}\zz^+ = \sum_{\lambda_r \in \Lambda_{uv}^+} F_r\ee_u \quad \text{and} \quad \zz^- = \sum_{\lambda_r \in \Lambda_{uv}^-} F_r \ee_u. \label{eq:1}\end{align}
Because of Theorem \ref{pst-char}.(\ref{pst-char:item1}), it follows that $\ee_u = \zz^+ + \zz^-$ and $\ee_v = \zz^+ - \zz^-$. Hence
\begin{align}\zz^+ = \frac{1}{2}\left(\ee_u + \ee_v\right) \quad \text{and} \quad \zz^- =\frac{1}{2}\left(\ee_u - \ee_v\right). \label{eq:2}\end{align}
The eigenvalue $0$ is in $\Lambda_{uv}^+$. Suppose it is the only eigenvalue in $\Lambda_{uv}^+$. Its corresponding eigenspace is spanned by $\j$, thus $\zz^+ = (1/n)\j$, which is only possible if $n = 2$. 

Now suppose that there is only one eigenvalue $\lambda$ in $\Lambda_{uv}^-$. Then $\zz^-$ is an eigenvector for $\lambda$. As a consequence, any vertex $w \neq v$ that is a neighbour of $u$ must also be a neighbour of $v$ and vice versa. Thus $u$ and $v$ are twins.
\end{proof}

We now proceed to show that some twins cannot be involved in Laplacian perfect state transfer.

\begin{thm}\label{lem:twins}
Suppose $X$ is neither the cycle on four vertices nor the complete graph $K_4$ minus one edge. If $u$ and $v$ are twins sharing exactly either one or two neighbours, then Laplacian perfect state transfer does not happen between $u$ and $v$.
\end{thm}
\begin{center}
\includegraphics{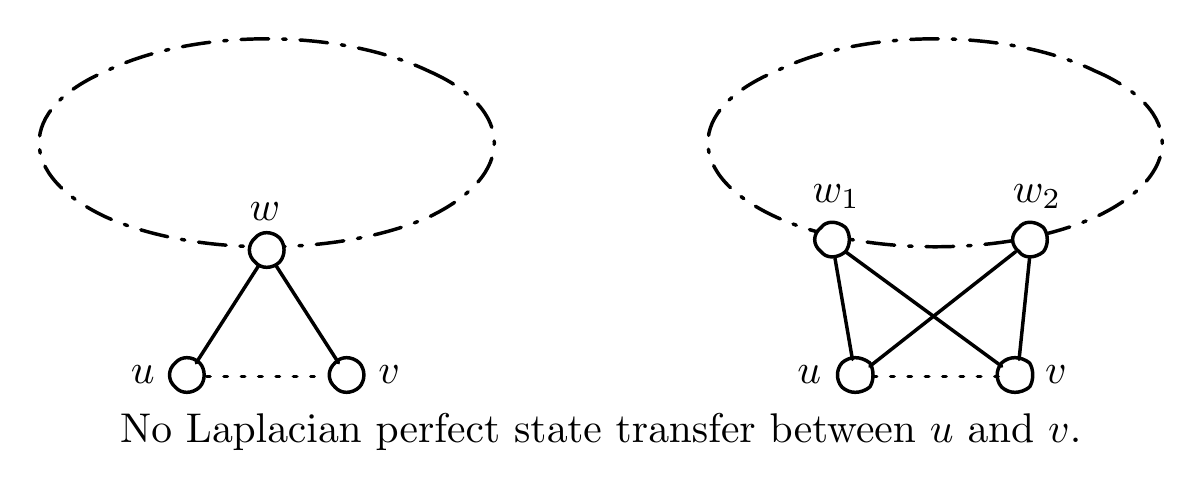}
\end{center}
\begin{proof}
Let $\sigma = 0$ if $u$ and $v$ are not neighbours, and $\sigma = 1$ if they are neighbours. Let $S$ be the set of common neighbours of $u$ and $v$, and $k = |S|$. By hypothesis, note that $k=1$ or $k=2$. Suppose Laplacian perfect state transfer happens between $u$ and $v$. Because $u$ and $v$ are twins, the vector that assigns $+1$ to $u$, $-1$ to $v$ and $0$ to all other vertices of the graph is an eigenvector of $L$ with eigenvalue $(k + 2 \sigma)$. All other eigenvectors are orthogonal to this one, thus  $F_r \ee_v = - F_r \ee_u$ if and only if $\lambda_r = k + 2 \sigma$.

Let $\ee_S$ be the characteristic vector of $S$. For all $\lambda_r \in \Lambda_{uv}^+$, it follows from $L (F_r\ee_u) = \lambda_r (F_r \ee_u)$ that
\[\ee_S^T F_r \ee_u = (k-\lambda_r) \ee_u^T F_r \ee_u.\]
From (\ref{eq:1}) and (\ref{eq:2}), we have
\[\sum_{\lambda_r \in \Lambda_{uv}^+} \ee_S F_r \ee_u = 0 \quad \text{and} \quad \sum_{\lambda_r \in \Lambda_{uv}^+} \ee_u F_r \ee_u = \frac{1}{2}.\]
Recall $\lambda_0 = 0 \in \Lambda_{uv}^+$ with $F_0 = (1/n) J$ and that $(k + 2 \sigma) \in \Lambda_{uv}^-$. By Theorem \ref{pst-char}, if $\lambda \in \Lambda_{uv}^+$ with $\lambda \neq 0$, then $\lambda$ is an integer and the power of two in the factorization of $\lambda$ is larger than the power of two in the factorization of $(k + 2 \sigma)$. In particular, it turns out that $\lambda \geq k+1$. So we have
\begin{align*}
0 = \sum_{\lambda_r \in \Lambda_{uv}^+} \ee_S F_r \ee_u & = \sum_{\lambda_r \in \Lambda_{uv}^+} (k-\lambda_r) \ee_u^T F_r \ee_u \\
& \leq \frac{1+k}{n} - \sum_{\lambda_r \in \Lambda_{uv}^+} \ee_u^T F_r \ee_u \\
& = \frac{1+k}{n} - \frac{1}{2}.
\end{align*}
If $k=1$, then $3\leq n \leq 4$. It is easy to manually check that Laplacian perfect state transfer does not happen in $K_3$, $P_3$, $K_{1,3}$ and $K_{1,3}$ plus one edge.

If $k=2$, then $4 \leq n \leq 6$. Computations carried out in \texttt{SAGE} show that the only graphs in this case in which Laplacian perfect state transfer occurs are the cycle on four vertices and $K_4$ minus one edge.
\end{proof}

If $X$ is a tree, twin vertices in $X$ must have a unique neighbour. As a consequence of our work in this section, we have the following corollary.

\begin{cor}\label{cor:twinstrees}
If Laplacian perfect state transfer happens in a tree between $u$ and $v$, then $|\Lambda_{uv}^-| \geq 2$.
\end{cor}

\section{No Laplacian Perfect State Transfer in Trees}

Recall a classical result in graph theory.

\begin{thm}[Matrix-Tree Theorem] \label{matrixtree}
Let $X$ be a graph on $n$ vertices, and $u$ be any of its vertices. Let $L[u]$ denote the sub-matrix of $L$ obtained by deleting the row and the column indexed by $u$. Then the number of spanning trees of $X$ is equal to $\det L[u]$.
\end{thm}

We start showing a simple yet surprising application of Lemma \ref{cospec-sizes} below.

\begin{cor}\label{cor:oddodd}
If a graph $X$ has an odd number of vertices and an odd number of spanning trees, then Laplacian perfect state transfer does not happen in $X$.
\end{cor}
\begin{proof}
Let $u$ and $v$ be any two vertices of $X$. By Theorem \ref{pst-char}.(\ref{pst-char:item3}), the elements in $\Lambda_{uv}^+$ must be even in order for Laplacian perfect state transfer to happen between $u$ and $v$. 

Note that the choice of $u$ in Theorem \ref{matrixtree} is irrelevant, so by using the Laplace expansion of a determinant, it is easy to observe that the number of spanning trees in a graph with $n$ vertices is equal to the product of its non-zero Laplacian eigenvalues (with repetition) scaled by $(1/n)$.

Therefore all non-zero integral Laplacian eigenvalues of $X$ are going to be odd, thus $\Lambda_{uv}^+$ contains only the eigenvalue $\lambda_0 = 0$. By Lemma \ref{cospec-sizes}, Laplacian perfect state transfer cannot happen.
\end{proof}

Computations carried out in \texttt{SAGE} show that among the 853 connected graphs on seven vertices, 339 have an odd number of spanning trees. Graphs with an odd number of spanning trees are precisely those graphs that contain no non-empty bicycle and are also known as pedestrian graphs. See Berman \cite{BermanBicycles} for more details. An example of a construction of pedestrian graphs consists of taking the successive $1$-sums of odd cycles and edges.
\begin{center}
\includegraphics{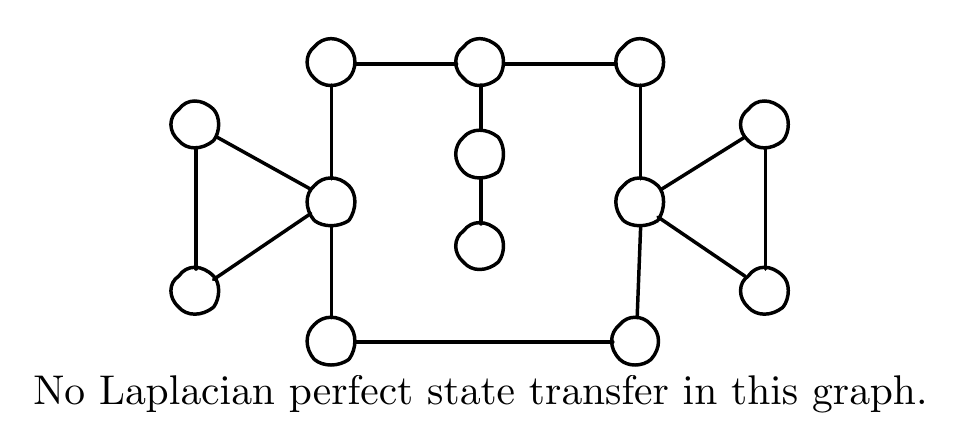}
\end{center}

Corollary \ref{cor:oddodd} rules out Laplacian perfect state transfer in trees with an odd number of vertices. We will now work on extending this result to all trees on more than two vertices. We will accomplish this by showing that integer eigenvalues are very bad candidates to belong to $\Lambda_{uv}^-$.

\begin{thm}\label{poweroftwo}
Suppose $X$ is a graph on $n$ vertices whose number of spanning trees is a power of two. Let $\lambda \neq 0$ be an integer eigenvalue of $X$. If $\lambda \in \Lambda_{uv}^-$, then $\lambda$ is a power of two.
\end{thm}
\begin{proof}
One of the alternative definitions for the rank of a matrix is that it is the largest order of any non-zero minor of the matrix. From Theorem \ref{matrixtree}, it follows that the rank of $L$ is equal to $(n-1)$ over any field $\Z_p$ with $p$ an odd prime.

Now let $\yy$ be an integer eigenvector for $\lambda$, chosen with the property that the greatest common divisor of its entries is equal to $1$. Suppose that there is an odd prime $p$ dividing $\lambda$. It follows that
\[L \yy \equiv 0 \pmod p.\]
Because the rank of $L$ over $\Z_p$ is $(n-1)$, the vector $\yy$ must be a scalar multiple of $\j$ over $\Z_p$, say $\yy \equiv k \j \pmod p$. If $\lambda \in \Lambda_{uv}^-$, the projection of $\ee_u$ onto any subspace of the $\lambda$-eigenspace must be the negative of the projection of $\ee_v$. Thus $\yy_u = -\yy_v$. Therefore
\[k \equiv -k \pmod p,\] 
and because $p$ is odd, the only possible solution is $k \equiv 0 \pmod p$. This contradicts the fact that the greatest common divisor of the entries of $\yy$ is equal to $1$.

Therefore if $\lambda \in \Lambda_{uv}^-$, no odd prime divides $\lambda$.
\end{proof}

\begin{cor}
No tree on more than two vertices admits Laplacian perfect state transfer.
\end{cor}
\begin{proof}
Let $u$ and $v$ be arbitrary vertices of the tree. Suppose that there are two distinct Laplacian eigenvalues $\lambda$ and $\mu$ belonging to  $\Lambda_{uv}^-$. If $g$ is the greatest common divisor of all non-zero elements in $\Lambda_u$, then Theorem \ref{pst-char}.(\ref{pst-char:item3}) implies that both $(\lambda / g)$ and $(\mu / g)$ must be odd integers in order for Laplacian perfect state transfer to happen. However Theorem \ref{poweroftwo} implies that $\lambda$ and $\mu$ are two distinct powers of two. Thus if Laplacian state transfer occurs, there can only be one eigenvalue in $\Lambda_{uv}^-$. But that is not possible according to Corollary \ref{cor:twinstrees}.
\end{proof}

\section{No Laplacian Perfect State Transfer in Other Cases}

We remark that our technology can be applied to other situations. The following result is an example.

\begin{cor}
  Suppose $X$ is a graph on $n$ vertices, $n > 4$. Suppose that the number of spanning trees of $X$ is a power of  two, and that $X$ admits Laplacian perfect state transfer between $u$ and $v$. Then $u$ and $v$ are twins with at least three common neighbours. Moreover, let $k$ be the number of common neighbours of $u$ and $v$. If $u \nsim v$, then $k$ is a power of two. If $u \sim v$, then $k+2$ is a power of two.
\end{cor}
\begin{proof}
By Theorem \ref{poweroftwo}, all eigenvalues in $\Lambda_{uv}^-$ are powers of two. By Theorem \ref{pst-char}.(\ref{pst-char:item3}), it follows that $|\Lambda_{uv}^-| = 1$. From Lemma \ref{cospec-sizes}, it follows that $u$ and $v$ are twins, and by Theorem \ref{lem:twins}, they have at least three common neighbours. Finally, let $\lambda$ be the eigenvalue in $\Lambda_{uv}^-$. By Theorem \ref{poweroftwo}, $\lambda$ is a power of two. If $u \nsim v$, then $\lambda = k$. If $u \sim v$, then $\lambda = k+2$.
\end{proof}

Computations carried out in \texttt{SAGE} show that among the 853 connected graphs on seven vertices, the number of spanning trees of 83 of them is a power of two. Among these, 58 contain twin vertices with one or two neighbours in common. There are 11117 connected graphs on eight vertices. The number of spanning trees of 360 of them is a power of two. The corollary above rules Laplacian perfect state transfer in at least one pair of vertices on 247 graphs among them.
\begin{center}
\includegraphics{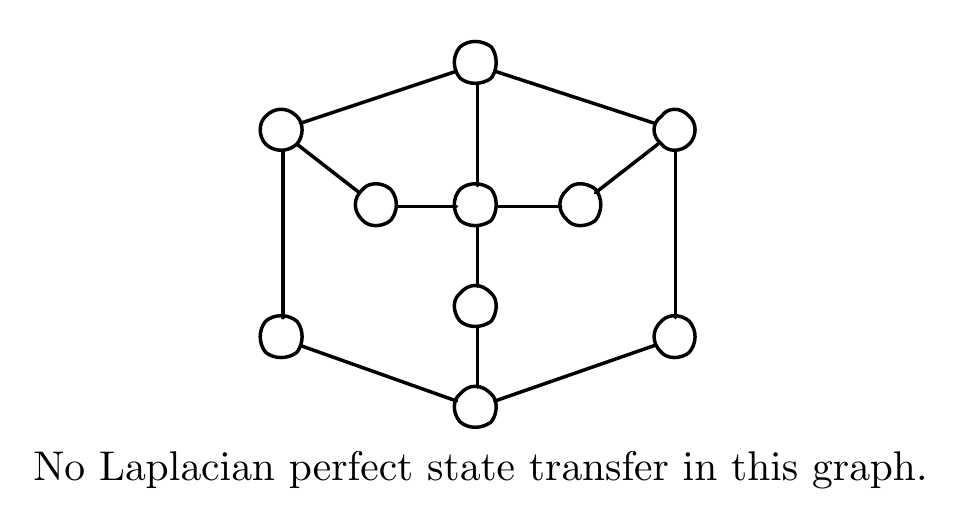}
\end{center}

We also make an observation that might be useful to rule out Laplacian perfect state transfer in bipartite graphs.

\begin{thm}
  If $X$ is a bipartite graph in which Laplacian perfect state
  transfer occurs, then the largest Laplacian eigenvalue $\lambda$ of
  $X$ must be an integer. Moreover, $u$ and $v$ are in the same colour
  class if and only if $\lambda \in \Lambda_{uv}^+$.
\end{thm}
\begin{proof}
  First note that the Laplacian $L = D - A$ is similar to the signless
  Laplacian $Q = D + A$. More specifically, if $\Sigma$ is the
  diagonal matrix where $\Sigma_{u,u} = \pm 1$ depending on the colour
  class of vertex $u$, then $\Sigma L \Sigma^{-1} = Q$. The signless
  Laplacian is a non-negative matrix, and it is irreducible. Hence we
  can apply the Perron-Frobenius theory (\cite[Section
  2.2]{BrouwerHaemers}) to argue that the $\lambda$-eigenspace of $L$
  is one-dimensional and spanned by an everywhere non-zero vector which
  is positive on one colour class and negative on the other. From
  Theorem \ref{pst-char}, it follows that $\lambda$ must be an integer, and
  belongs to $\Lambda_{uv}^+$ if and only if the entries corresponding
  to $u$ and $v$ in its eigenvector have the same sign.
\end{proof}

We checked in \texttt{SAGE} that among the 182 bipartite graphs on
eight vertices, the largest Laplacian eigenvalue is an integer in only
10 of them.

\section{Adjacency Perfect State Transfer in Bipartite Graphs}

We denote the spectral decomposition of the adjacency matrix of a graph $X$ by
\[A = \sum_{r = 0}^d \theta_r E_r,\]
with the understanding that $\theta_0 > ... > \theta_d$. Given two vertices $u$ and $v$, we say that an eigenvalue $\theta_r$ is in the \textit{eigenvalue support} of $u$ if $E_r \ee_u \neq 0$. We will denote the eigenvalue support of $u$ by $\Phi_u$. We also define $\Phi_{uv}^+ \subset \Phi_u$ to be such that $\theta_r \in \Phi_{uv}^+$ if and only if $E_r \ee_u = E_r \ee_v$, and correspondingly $\Phi_{uv}^-$ to be such that $\theta_r \in \Phi_{uv}^-$ if and only if $E_r \ee_u = - E_r \ee_v$. 

Similarly to what happened in Theorem \ref{pst-char}, a necessary condition for perfect state transfer between $u$ and $v$ here is that $E_r \ee_u = \pm E_r \ee_v$ for all $r$, or equivalently, $\Phi_u = \Phi_{uv}^+ \cup \Phi_{uv}^-$. Vertices $u$ and $v$ satisfying this condition with respect to the adjacency matrix are called \textit{strongly cospectral}.

The following result characterizes the nature of eigenvalues in the support of vertices involved in perfect state transfer.

\begin{thm}[Godsil \cite{GodsilPerfectStateTransfer12}, Theorem 6.1] \label{esupportgodsil}
If $X$ admits perfect state transfer between $u$ and $v$, then the non-zero elements in $\Phi_u$ are either all integers or all quadratic integers. Moreover, there is a square-free integer $\Delta$, an integer $a$, and integers $b_r$ such that
\[\theta_r \in \Phi_u \implies \theta_r = \frac{1}{2} \big( a + b_r \sqrt{\Delta} \big).\]
Here we allow $\Delta = 1$ for the case where all eigenvalues are integers, and $a =0$ for the case where they are all integer multiples of $\sqrt{\Delta}$. \qed
\end{thm}

If the eigenvalues are integers, the following result is the adjacency matrix version of Theorem \ref{pst-char}.
\begin{thm} \label{pst-char-adj}
Let $X$ be a graph and $u$ and $v$ two of its vertices. Suppose the elements of $\Phi_u$ are integers. Then there is perfect state transfer between $u$ and $v$ at time $t \in \R^+$ and phase $\gamma \in \C$ if and only if the following conditions hold.
\begin{enumerate}[(i)]
\item Vertices $u$ and $v$ are strongly cospectral. \label{pst-char-adj-item1}
\item \label{pst-char-adj-item2} Let $\displaystyle g = \gcd \left\{ \theta_0 - \theta_r  : \theta_r \in \Phi_u\right\}$. Then 
\begin{enumerate}[a)]
\item $\theta_r \in \Phi_{uv}^+$ if and only if $\dfrac{\theta_0 - \theta_r}{g}$ is even, and
\item $\theta_r \in \Phi_{uv}^-$ if and only if $\dfrac{\theta_0 - \theta_r}{g}$ is odd.
\end{enumerate}
\end{enumerate} 
Moreover, $t$ is an odd multiple of $(\pi / g)$, and $\gamma = \e^{\ii t \theta_0}$.
\qed
\end{thm}

Suppose $X$ is a bipartite graph, in which case the adjacency matrix of $X$ can be written as
\begin{align} A(X) = \begin{pmatrix}0 & B \\ B^T & 0 \end{pmatrix}. \label{bipgraphs:eq:bipartitionofA} \end{align}
If $\zz$ is an eigenvector for $A(X)$, we split it according to the blocks of $A$ as $\zz = (\zz_1 , \zz_2)$. This immediately implies that if $\theta$ is an eigenvalue for a bipartite graph $X$ with corresponding eigenvector $(\zz_1,\zz_2)$, then $-\theta$ is an eigenvalue with eigenvector $(\zz_1,-\zz_2)$. As a consequence, we have the following result.

\begin{lem}\label{bipgraphs:lem:esupport}
If $X$ is a bipartite graph and $u \in V(X)$ is involved in perfect state transfer, then no eigenvalue in the support of $u$ is of the form $\frac{a + b \sqrt{\Delta}}{2}$ for non-zero integers $a$ and $b$ with $\Delta$ a square-free larger than $1$.
\end{lem}
\begin{proof}
Suppose $\theta = \frac{a + b \sqrt{\Delta} }{2}$ is in the support of $u$. Then its algebraic conjugate $\overline{\theta} = \frac{a - b \sqrt{\Delta}}{2}$ is also in the support, and by the observation above, the values $-\theta$ and $-\overline{\theta}$ are also eigenvalues in the support of $u$. The ratio condition (Godsil \cite[Theorem 2.2]{GodsilPeriodicGraphs11}) states that
\[\frac{\theta - \overline{\theta}}{\theta - (-\overline{\theta})} \in \Q,\] a contradiction.
\end{proof}

\section{No Perfect State Transfer in Certain Bipartite Graphs}

For more details about the next result, see Godsil \cite{GodsilInversesTrees}.

\begin{thm} \label{bipgraphs:thm:invertiblegraphs}
If $X$ is a bipartite graph with a unique perfect matching, then $A(X)$ is invertible and its inverse is an integer matrix. If $X$ is a tree, then $A(X)$ is invertible if and only if $X$ has a (unique) perfect matching.
\end{thm}

As a consequence of the theorem above, we have the following.

\begin{thm}\label{bipgraphs:thm:nopstinvertible}
Except for the path on two vertices, no connected bipartite graph with a unique perfect matching admits perfect state transfer.
\end{thm}
\begin{proof}
Suppose $X$ is a bipartite graph with a unique perfect matching, and that $u$ is involved in perfect state transfer. Let $\theta$ be an eigenvalue in the support of $u$, and recall from Theorem \ref{esupportgodsil} that $\theta$ is a quadratic integer. By Theorem \ref{bipgraphs:thm:invertiblegraphs}, $(1/\theta) $ must be an algebraic integer, and so $\theta$ is either $+1$, $-1$, or of the form $(a + b \sqrt{\Delta})/2$ \ with $a$ and $b$ non-zero. Lemma \ref{bipgraphs:lem:esupport} excludes the latter case, and hence the only eigenvalues in the support of $u$ are $+1$ and $-1$. It is easy to see in this case that the connected component containing $u$ is equal to $P_2$, and so the result follows.
\end{proof}

The result above allows us to rule out perfect state transfer for a large class of trees. We can work a bit more in the case where perfect state transfer happens between vertices in different classes of the bipartition. 

\begin{lem} \label{bipgraphs:lem:difclassessupint}
If $X$ is a bipartite graph admitting perfect state transfer between $u$ and $v$, and if $u$ and $v$ are in different classes, then their support contains only integer eigenvalues.
\end{lem}
\begin{proof}
From Lemma \ref{bipgraphs:lem:esupport}, we need only to consider the case where $b\sqrt{\Delta}$ is in the support of $u$. Because $-b\sqrt{\Delta}$ is the algebraic conjugate of $b \sqrt{\Delta}$, it follows that any eigenvector for $b\sqrt{\Delta}$ can be partitioned as $(\zz_1,\sqrt{\Delta} \zz_2)$, where $\zz_1$ and $\zz_2$ are rational vectors. As a consequence, the absolute value of the entries in the $u$th and $v$th position are different, and so these vertices cannot be strongly cospectral, a necessary condition for perfect state transfer.
\end{proof}

Now recall Equation \ref{bipgraphs:eq:bipartitionofA}, and observe that it implies that
\[U_A(t) = \left(\begin{array}{cc}
\cos(t\sqrt{BB^T}) & \ii \sin(t \sqrt{BB^T}) B \\ \ii \sin(t \sqrt{B^TB}) B^T & \cos(t \sqrt{B^TB})\end{array}\right).\]
As a consequence, if perfect state transfer happens in a bipartite graph between vertices in different classes, it must happen with phase $\pm \ii$. We use that to prove the following result.

\begin{thm}
Suppose $X$ is bipartite, perfect state transfer happens between $u$ and $v$ at time $t$, and $u$ and $v$ belong to different classes. Then the powers of two in the factorizations of the eigenvalues in the support of $u$ are all equal. In particular, $0$ cannot be in the support of $u$.
\end{thm}
\begin{proof}
We saw that perfect state transfer must happen in this case with phase $\pm \ii$. Let $2^\alpha$ be the largest power of two dividing $\theta_0$. It follows from Theorem \ref{pst-char-adj} that $\e^{\ii t \theta_0} = \pm \ii$, thus $t$ is an odd multiple of $(\pi/ 2^{\alpha+1})$. Let $\theta_r$ be an eigenvector in the support of $u$, and denote $\theta_{-r} = - \theta_r$.

Because $u$ and $v$ are in different classes (but are strongly cospectral), $E_r \ee_u = \sigma E_r \ee_v$ and $E_{-r} \ee_u = -\sigma E_{-r} \ee_v$ with $\sigma = \pm 1$, and so $t (\theta_0 - \sigma\theta_r)$ is an even multiple of $\pi$, whereas $t (\theta_0 + \sigma\theta_r)$ is an odd multiple of $\pi$. All together, we have the following three equations:
\begin{align*}
\theta_0 & \equiv 2^{\alpha} \pmod {2^{\alpha+1} },\\
\theta_0 - \sigma\theta_r & \equiv 0 \pmod {2^{\alpha+2} },\\
\theta_0+\sigma\theta_r & \equiv 2^{\alpha+1} \pmod{ 2^{\alpha+2}}.
\end{align*}

From that it follows that $\theta_r$ is also congruent to $2^\alpha \pmod{2^{\alpha+1}}$, and that $0$ cannot be in the support of $u$.
\end{proof}

All the results in this paper suggest that very restrictive conditions must hold for perfect state transfer to happen in bipartite graphs. In the context of trees, we could successfully rule out Laplacian perfect state transfer. For the adjacency case, we showed that all trees with a perfect matching do not admit perfect state transfer. We carried out computations in \texttt{SAGE} to check all trees up to ten vertices and we did not find any examples of perfect state transfer except if the tree is $P_2$ or $P_3$. We therefore propose the following conjecture.

\begin{conjecture}
No tree except for $P_2$ and $P_3$ admits (adjacency) perfect state transfer.
\end{conjecture}

\section*{Acknowledgement}

Both authors are grateful to Chris Godsil for extremely helpful
discussions about quantum walks in graphs and for asking the question
of whether perfect state transfer happens in trees. This work was done
while G.C. was a PhD student and H.L. was an Undergraduate Research
Assistant working under Chris's supervision.

\bibliography{qwalks}{}
\bibliographystyle{plain}

\end{document}